\documentclass[10pt]{article}

\usepackage[margin=1in]{geometry}
\usepackage{amsmath,amssymb,amsthm}
\usepackage{enumitem}
\usepackage[colorlinks=true,citecolor=blue,linkcolor=blue,urlcolor=blue]{hyperref}

\setlist[itemize]{itemsep=2pt,topsep=3pt}
\setlist[enumerate]{itemsep=2pt,topsep=3pt}

\newtheorem{theorem}{Theorem}[section]
\newtheorem{lemma}[theorem]{Lemma}
\newtheorem{proposition}[theorem]{Proposition}
\newtheorem{corollary}[theorem]{Corollary}
\newtheorem{problem}[theorem]{Problem}

\theoremstyle{definition}
\newtheorem{definition}[theorem]{Definition}
\newtheorem{algorithm}[theorem]{Algorithm}
\newtheorem{example}[theorem]{Example}

\theoremstyle{remark}
\newtheorem{remark}[theorem]{Remark}

\DeclareMathOperator{\spanop}{span}
\DeclareMathOperator{\rank}{rank}
\DeclareMathOperator{\codim}{codim}

\title{Extremal Subspace Covers in Finite Vector Spaces}

\author{
Mohsen Aliabadi\\
Department of Mathematics, Clayton State University\\
Morrow, GA, USA\\
\texttt{maliabadi@clayton.edu}\\
\texttt{mohsenmath88@gmail.com}
}

\date{}

\begin{document}

\maketitle

\begin{abstract}
Let \(V\) be an \(n\)-dimensional vector space over the finite field
\(\mathbb F_q\), where \(n\ge 2\). It is classical that \(q+1\) proper
subspaces are necessary and sufficient to cover \(V\). In this paper we
study extremal refinements of this covering theorem.

For \(1\le m\le q+1\), we determine the maximum possible size of a union
of \(m\) proper subspaces of \(V\). We prove that
\[
\max_{W_1,\ldots,W_m<V}
\left|W_1\cup\cdots\cup W_m\right|
=
q^{n-2}\bigl(1+m(q-1)\bigr).
\]
We also classify all equality cases: for \(m\ge 2\), equality occurs
precisely when the subspaces are distinct hyperplanes containing one
common codimension-two subspace. As consequences, we obtain a structural
classification of minimum covers of \(V\) by proper subspaces and a sharp
defect estimate for unions of \(q\) proper subspaces.

We then introduce basis-blocking families. A family of proper subspaces
is basis-blocking if its union meets every basis of \(V\). We prove that
the least size of a basis-blocking family is \(q\), and we classify all
basis-blocking families of this minimum size. Finally, we give the affine
analogue and record an elementary recognition procedure for extremal
families. The results give a compact extremal-combinatorial description
of how proper subspaces cover finite vector spaces.
\end{abstract}



\section{Introduction and preliminaries}

A basic covering problem in finite geometry asks how efficiently a finite
vector space can be covered by proper subspaces. Over an infinite field,
a finite-dimensional vector space cannot be written as a finite union of
proper subspaces. Over finite fields, however, this statement fails in
the strongest possible way: if \(V\) is an \(n\)-dimensional vector space
over \(\mathbb F_q\), with \(n\ge 2\), then \(V\) can be covered by
exactly \(q+1\) proper subspaces, and no fewer.

This phenomenon is classical. Jamison studied coverings of finite fields
by cosets of subspaces \cite{Jamison}, and Brouwer and Schrijver studied
related blocking problems in affine spaces \cite{BrouwerSchrijver}.
Khare treated vector spaces as unions of proper subspaces in broader
generality, including fixed-codimension variants \cite{Khare}. Clark
later computed covering numbers and irredundant covering numbers in
linear algebra \cite{Clark}. Covering problems for modules and rings have
also been studied; see, for example, \cite{KhareTikaradze,SwartzWerner}.

The present paper focuses on a sharper extremal question. Rather than
asking only how many proper subspaces are needed to cover \(V\), we ask:
if \(m\) proper subspaces are allowed, how many vectors can they cover?
For \(1\le m\le q+1\), we determine the exact maximum and classify the
extremal configurations.

Our first main theorem is the following.

\begin{theorem}\label{thm:intro-max-union}
Let \(V\) be an \(n\)-dimensional vector space over \(\mathbb F_q\),
where \(n\ge 2\). If \(1\le m\le q+1\), then
\[
\max_{W_1,\ldots,W_m<V}
\left|W_1\cup\cdots\cup W_m\right|
=
q^{n-2}\bigl(1+m(q-1)\bigr).
\]
For \(m=1\), equality holds exactly when \(W_1\) is a hyperplane.
For \(m\ge 2\), equality holds exactly when
\(W_1,\ldots,W_m\) are distinct hyperplanes containing a common
codimension-two subspace of \(V\).
\end{theorem}

Thus the usual covering theorem is the endpoint \(m=q+1\) of a sharper
maximum-union theorem. The extremal objects are pencils of hyperplanes:
choose a codimension-two subspace \(U\), and take hyperplanes containing
\(U\).

Our second main object is a basis-blocking family.

\begin{definition}
A family of proper subspaces
\[
\mathcal W=\{W_1,\ldots,W_m\}
\]
of \(V\) is called \emph{basis-blocking} if every basis of \(V\)
contains at least one vector belonging to
\[
W_1\cup\cdots\cup W_m.
\]
Equivalently, \(V\setminus(W_1\cup\cdots\cup W_m)\) contains no basis of
\(V\).
\end{definition}

We determine the minimum size and classify all extremal basis-blocking
families.

\begin{theorem}\label{thm:intro-basis-blocking}
Let \(V\) be an \(n\)-dimensional vector space over \(\mathbb F_q\),
where \(n\ge 2\). The least number of proper subspaces whose union meets
every basis of \(V\) is \(q\).

Moreover, the basis-blocking families of size \(q\) are precisely the
following families: choose a codimension-two subspace \(U\le V\), take
the \(q+1\) hyperplanes containing \(U\), and omit one of them. The
remaining \(q\) hyperplanes form a minimum basis-blocking family.
\end{theorem}

Thus the covering number and the basis-blocking number differ by one:
\[
\text{to cover all vectors, one needs }q+1\text{ proper subspaces,}
\]
whereas
\[
\text{to block all bases, one needs }q\text{ proper subspaces.}
\]
Both extremal problems are governed by the same finite-geometric object:
a pencil of hyperplanes through a codimension-two subspace.

Throughout the paper, \(q\) denotes a prime power and \(\mathbb F_q\)
denotes the finite field with \(q\) elements. Unless otherwise stated,
\(V\) denotes an \(n\)-dimensional vector space over \(\mathbb F_q\), with
\(n\ge 2\).

We use the following terminology. A subspace \(W\le V\) is \emph{proper}
if \(W\ne V\). A family \(W_1,\ldots,W_m\) of proper subspaces is a
\emph{cover} of \(V\) if
\[
V=W_1\cup\cdots\cup W_m.
\]
Such a cover is \emph{irredundant} if no member can be removed while
still covering \(V\). The \emph{codimension} of \(W\) is
\[
\codim W=\dim V-\dim W.
\]
A \emph{hyperplane} is a subspace of codimension one. A \emph{pencil of
hyperplanes} means the family of all hyperplanes of \(V\) containing a
fixed codimension-two subspace \(U\le V\). If \(V=\mathbb F_q^n\) and
\(H=\ker \ell\), where \(0\ne \ell\in V^\ast\), then the
\emph{projective class} of \(\ell\), denoted \([\ell]\), is the
one-dimensional subspace of \(V^\ast\) spanned by \(\ell\). A
\emph{projective line} in \(\mathbb P(V^\ast)\) is the set of all
one-dimensional subspaces of a fixed two-dimensional subspace of
\(V^\ast\). The rank of a finite set of linear functionals is the
dimension of their span in \(V^\ast\).

The paper is organized as follows. Section 2 proves the maximum-union
theorem and its consequences for minimum covers. Section 3 treats
basis-blocking families. Section 4 gives the affine analogue, the
recognition criterion, examples, and open problems.

\section{Maximum unions and minimum covers}

We begin with the standard lower bound for covering \(V\) by proper
subspaces.

\begin{lemma}\label{lem:lower-bound}
Let \(V\) be a vector space over \(\mathbb F_q\), with \(\dim V\ge 2\).
If
\[
V=W_1\cup\cdots\cup W_m
\]
is a cover by proper subspaces, then
\[
m\ge q+1.
\]
\end{lemma}

\begin{proof}
After deleting redundant members, assume that the cover is irredundant.
Then \(m\ge 2\). Since the cover is irredundant, we may choose
\[
a\in W_1\setminus\bigcup_{i=2}^m W_i
\]
and
\[
b\in W_2\setminus
\bigcup_{\substack{1\le i\le m\\ i\ne 2}} W_i .
\]
For each \(t\in\mathbb F_q^\times\), consider the vector
\[
a+tb.
\]
It lies in neither \(W_1\) nor \(W_2\). Indeed, if \(a+tb\in W_1\), then
since \(a\in W_1\), we get \(tb\in W_1\), hence \(b\in W_1\), a
contradiction. Similarly, \(a+tb\notin W_2\).

Thus every vector \(a+tb\), with \(t\in\mathbb F_q^\times\), must lie in
one of \(W_3,\ldots,W_m\). We claim that each \(W_i\), \(i\ge 3\), can
contain at most one of these vectors. If
\[
a+t_1b,\ a+t_2b\in W_i
\]
with \(t_1\ne t_2\), then subtracting gives
\[
(t_1-t_2)b\in W_i.
\]
Hence \(b\in W_i\). Then
\[
a=(a+t_1b)-t_1b\in W_i,
\]
contradicting the choice of \(a\). Therefore the \(q-1\) vectors
\(a+tb\), \(t\in\mathbb F_q^\times\), must be assigned to distinct
members among \(W_3,\ldots,W_m\). Hence
\[
m-2\ge q-1,
\]
and therefore \(m\ge q+1\).
\end{proof}

The following construction attains the bound and will also produce the
extremal examples for the maximum-union theorem.

\begin{proposition}\label{prop:pencil}
Let \(U\le V\) have codimension two. Then the hyperplanes of \(V\)
containing \(U\) form a family of \(q+1\) hyperplanes whose union is
\(V\). Moreover, if \(m\) of these hyperplanes are chosen, where
\(1\le m\le q+1\), then their union has size
\[
q^{n-2}\bigl(1+m(q-1)\bigr).
\]
\end{proposition}

\begin{proof}
Since \(\dim(V/U)=2\), the hyperplanes of \(V\) containing \(U\)
correspond exactly to the one-dimensional subspaces of \(V/U\). There
are \(q+1\) such subspaces, and they cover \(V/U\). Taking inverse images
under the quotient map \(V\to V/U\), the \(q+1\) corresponding
hyperplanes cover \(V\).

Now choose \(m\) of these hyperplanes. Any two of them meet exactly in
\(U\). Hence their parts outside \(U\) are pairwise disjoint. Therefore
their union has size
\[
|U|+m(q^{n-1}-|U|)
=
q^{n-2}+m(q^{n-1}-q^{n-2})
=
q^{n-2}\bigl(1+m(q-1)\bigr).
\]
\end{proof}

We now prove the main extremal result.

\begin{theorem}\label{thm:max-union}
Let \(V\) be an \(n\)-dimensional vector space over \(\mathbb F_q\),
where \(n\ge 2\). If
\[
W_1,\ldots,W_m
\]
are proper subspaces of \(V\) and \(1\le m\le q+1\), then
\[
\left|W_1\cup\cdots\cup W_m\right|
\le
q^{n-2}\bigl(1+m(q-1)\bigr).
\]
The bound is sharp.

For \(m=1\), equality holds if and only if \(W_1\) is a hyperplane. For
\(m\ge 2\), equality holds if and only if \(W_1,\ldots,W_m\) are
distinct hyperplanes containing a common codimension-two subspace of
\(V\).
\end{theorem}

\begin{proof}
For each \(i\), choose a hyperplane \(H_i\) containing \(W_i\). Then
\[
W_1\cup\cdots\cup W_m
\subseteq
H_1\cup\cdots\cup H_m.
\]
Thus it suffices to prove the upper bound for hyperplanes.

The first hyperplane contributes \(q^{n-1}\) points. For each \(j\ge 2\),
the new contribution of \(H_j\) is at most
\[
|H_j|-|H_j\cap H_1|.
\]
If \(H_j\ne H_1\), then \(H_j\cap H_1\) has codimension two in \(V\),
and hence has size \(q^{n-2}\). If \(H_j=H_1\), the new contribution is
even smaller. Therefore each additional hyperplane contributes at most
\[
q^{n-1}-q^{n-2}
\]
new points. Hence
\[
\left|H_1\cup\cdots\cup H_m\right|
\le
q^{n-1}+(m-1)(q^{n-1}-q^{n-2})
=
q^{n-2}\bigl(1+m(q-1)\bigr).
\]
The same upper bound holds for \(W_1,\ldots,W_m\).

Sharpness follows from Proposition \ref{prop:pencil}: choose any \(m\)
hyperplanes from the pencil of hyperplanes containing a fixed
codimension-two subspace.

It remains to analyze equality. The case \(m=1\) is immediate: a proper
subspace has at most \(q^{n-1}\) elements, with equality precisely when
it is a hyperplane.

Assume now that \(m\ge 2\), and suppose equality holds. Replacing
\(W_i\) by \(H_i\) can only increase the union, so equality forces
\[
\left|H_1\cup\cdots\cup H_m\right|
=
q^{n-2}\bigl(1+m(q-1)\bigr).
\]
Thus equality holds at every step of the counting argument above.

In particular, the hyperplanes \(H_1,\ldots,H_m\) are distinct, and for
every \(j\ge 2\),
\[
H_j\cap(H_1\cup\cdots\cup H_{j-1})
=
H_j\cap H_1.
\]
Fix \(j\ge 3\). From this equality we get
\[
H_j\cap H_2\subseteq H_j\cap H_1.
\]
Now \(H_j\cap H_1\) and \(H_j\cap H_2\) are both intersections of two
distinct hyperplanes, so each has codimension two in \(V\), equivalently
dimension \(n-2\). Since one is contained in the other and both have the
same dimension, they are equal:
\[
H_j\cap H_2=H_j\cap H_1.
\]
Both therefore coincide with \(H_1\cap H_2\), which also has dimension
\(n-2\). Writing
\[
U=H_1\cap H_2,
\]
we conclude that \(U\subseteq H_j\) for every \(j\). Thus
\(H_1,\ldots,H_m\) are distinct hyperplanes through the common
codimension-two subspace \(U\).

Finally, since
\[
W_1\cup\cdots\cup W_m=H_1\cup\cdots\cup H_m,
\]
we prove that \(W_i=H_i\) for every \(i\). In a pencil of hyperplanes
through \(U\), the sets \(H_i\setminus U\) are pairwise disjoint. Hence a
point of \(H_i\setminus U\) cannot belong to \(H_j\), and therefore
cannot belong to \(W_j\), for \(j\ne i\). Since the two unions are equal,
every point of \(H_i\setminus U\) must lie in \(W_i\). Therefore
\[
H_i\setminus U\subseteq W_i.
\]
But \(H_i\setminus U\) spans \(H_i\). Hence \(W_i=H_i\). This proves the
equality characterization.

Conversely, if \(W_1,\ldots,W_m\) are distinct hyperplanes containing a
common codimension-two subspace, Proposition \ref{prop:pencil} shows that
equality holds.
\end{proof}

The endpoint \(m=q+1\) gives the structure of minimum covers.

\begin{corollary}\label{cor:minimum-covers}
Let \(V\) be an \(n\)-dimensional vector space over \(\mathbb F_q\),
where \(n\ge 2\). Suppose
\[
V=W_1\cup\cdots\cup W_{q+1},
\]
where each \(W_i\) is a proper subspace of \(V\). Then
\(W_1,\ldots,W_{q+1}\) are exactly the \(q+1\) hyperplanes containing a
fixed codimension-two subspace of \(V\).

Conversely, for any codimension-two subspace \(U\le V\), the \(q+1\)
hyperplanes containing \(U\) form a minimum cover of \(V\) by proper
subspaces.
\end{corollary}

\begin{proof}
Since the union is all of \(V\), its size is \(q^n\). For \(m=q+1\), the
upper bound in Theorem \ref{thm:max-union} becomes
\[
q^{n-2}\bigl(1+(q+1)(q-1)\bigr)=q^{n-2}q^2=q^n.
\]
Thus equality holds in Theorem \ref{thm:max-union}, and the equality
classification gives the result. The converse follows from Proposition
\ref{prop:pencil} and Lemma \ref{lem:lower-bound}.
\end{proof}

The case \(m=q\) gives a sharp defect estimate.

\begin{corollary}\label{cor:defect}
Let \(V\) be an \(n\)-dimensional vector space over \(\mathbb F_q\),
where \(n\ge 2\). If
\[
W_1,\ldots,W_q
\]
are proper subspaces of \(V\), then
\[
\left|V\setminus(W_1\cup\cdots\cup W_q)\right|
\ge
(q-1)q^{n-2}.
\]
Equality holds if and only if \(W_1,\ldots,W_q\) are \(q\) hyperplanes
from a pencil of \(q+1\) hyperplanes through a common codimension-two
subspace.
\end{corollary}

\begin{proof}
By Theorem \ref{thm:max-union},
\[
|W_1\cup\cdots\cup W_q|
\le
q^{n-2}\bigl(1+q(q-1)\bigr).
\]
Therefore
\[
\left|V\setminus(W_1\cup\cdots\cup W_q)\right|
\ge
q^n-q^{n-2}\bigl(1+q(q-1)\bigr)
=
(q-1)q^{n-2}.
\]
The equality statement follows from Theorem \ref{thm:max-union}.
\end{proof}

\begin{remark}
Corollary \ref{cor:defect} quantifies exactly how close \(q\) proper
subspaces can come to covering \(V\). They cannot cover \(V\), but the
best possible construction misses only \((q-1)q^{n-2}\) vectors. The
extremal configuration is obtained by taking all but one hyperplane from
a pencil through a codimension-two subspace.
\end{remark}

\section{Basis-blocking families}

We now turn to the weaker condition of blocking all bases.

\begin{lemma}\label{lem:span-criterion}
Let \(X\subseteq V\). Then \(X\) contains a basis of \(V\) if and only if
\[
\spanop(X)=V.
\]
Consequently, \(W_1,\ldots,W_m\) are basis-blocking if and only if
\[
\spanop\bigl(V\setminus(W_1\cup\cdots\cup W_m)\bigr)\ne V.
\]
\end{lemma}

\begin{proof}
If \(X\) contains a basis of \(V\), then \(X\) spans \(V\). Conversely,
if \(X\) spans \(V\), then a standard basis-extraction argument produces
a basis of \(V\) contained in \(X\).
\end{proof}

\begin{theorem}\label{thm:basis-blocking-number}
Let \(V\) be an \(n\)-dimensional vector space over \(\mathbb F_q\),
where \(n\ge 2\). The least number of proper subspaces whose union meets
every basis of \(V\) is \(q\).
\end{theorem}

\begin{proof}
First we show that \(q-1\) proper subspaces cannot block all bases. Let
\[
W_1,\ldots,W_m
\]
be proper subspaces with \(m\le q-1\). We construct a basis avoiding all
of them.

Suppose that \(v_1,\ldots,v_r\), with \(0\le r<n\), have already been
chosen linearly independently and outside
\[
W_1\cup\cdots\cup W_m.
\]
Let
\[
U_r=\spanop(v_1,\ldots,v_r).
\]
Then \(U_r\) is a proper subspace of \(V\). The family
\[
U_r,W_1,\ldots,W_m
\]
contains at most \(q\) proper subspaces. By Lemma \ref{lem:lower-bound},
these subspaces cannot cover \(V\). Hence there exists
\[
v_{r+1}\in V\setminus(U_r\cup W_1\cup\cdots\cup W_m).
\]
Then \(v_1,\ldots,v_{r+1}\) are linearly independent, and the new vector
avoids every \(W_i\). Repeating this argument gives a basis of \(V\)
avoiding \(W_1,\ldots,W_m\). Therefore \(q-1\) proper subspaces are not
basis-blocking.

Now we show that \(q\) subspaces can block all bases. Choose a
codimension-two subspace \(U\le V\), and let
\[
H_0,H_1,\ldots,H_q
\]
be the \(q+1\) hyperplanes containing \(U\). The union of all these
hyperplanes is \(V\). Hence
\[
V\setminus(H_1\cup\cdots\cup H_q)\subseteq H_0.
\]
Since \(H_0\) is a proper subspace, the complement
\(V\setminus(H_1\cup\cdots\cup H_q)\) cannot contain a basis of \(V\).
Thus \(H_1,\ldots,H_q\) are basis-blocking. The minimum size is therefore
\(q\).
\end{proof}

We now classify the minimum basis-blocking families.

\begin{theorem}\label{thm:basis-blocking-classification}
Let \(V\) be an \(n\)-dimensional vector space over \(\mathbb F_q\),
where \(n\ge 2\). Let
\[
W_1,\ldots,W_q
\]
be proper subspaces of \(V\). Then \(W_1,\ldots,W_q\) are
basis-blocking if and only if there exists a codimension-two subspace
\(U\le V\) such that \(W_1,\ldots,W_q\) are \(q\) of the \(q+1\)
hyperplanes containing \(U\).
\end{theorem}

\begin{proof}
The ``if'' direction follows from the second half of the proof of
Theorem \ref{thm:basis-blocking-number}.

Conversely, suppose \(W_1,\ldots,W_q\) are basis-blocking. Set
\[
X=V\setminus(W_1\cup\cdots\cup W_q).
\]
By Lemma \ref{lem:span-criterion},
\[
\spanop(X)\ne V.
\]
Choose a hyperplane \(H_0\) containing \(\spanop(X)\). Then
\[
V=H_0\cup W_1\cup\cdots\cup W_q.
\]
This expresses \(V\) as a union of \(q+1\) proper subspaces. These
\(q+1\) subspaces must in fact be distinct. If two of them coincided,
then \(V\) would be covered by at most \(q\) proper subspaces,
contradicting Lemma \ref{lem:lower-bound}. Therefore
\[
H_0,W_1,\ldots,W_q
\]
form a cover of \(V\) by exactly \(q+1\) proper subspaces. By Corollary
\ref{cor:minimum-covers}, they are precisely the \(q+1\) hyperplanes
through a common codimension-two subspace \(U\). Hence
\(W_1,\ldots,W_q\) are \(q\) of the \(q+1\) hyperplanes containing \(U\).
\end{proof}

\begin{corollary}\label{cor:basis-avoidance}
Let \(V\) be an \(n\)-dimensional vector space over \(\mathbb F_q\),
where \(n\ge 2\), and let
\[
W_1,\ldots,W_m
\]
be proper subspaces of \(V\). If \(m\le q-1\), then there exists a basis
\[
v_1,\ldots,v_n
\]
of \(V\) such that
\[
v_j\notin W_1\cup\cdots\cup W_m
\]
for every \(1\le j\le n\). The bound \(q-1\) is sharp.
\end{corollary}

\begin{proof}
The existence of such a basis is proved in the first part of Theorem
\ref{thm:basis-blocking-number}. Sharpness follows from Theorem
\ref{thm:basis-blocking-number}, since \(q\) proper subspaces can block
all bases.
\end{proof}

\begin{remark}
The covering number and the basis-blocking number differ by one:
\[
q+1\quad\text{proper subspaces are needed to cover }V,
\]
whereas
\[
q\quad\text{proper subspaces are needed to meet every basis of }V.
\]
The extremal configurations in both cases are governed by pencils of
hyperplanes.
\end{remark}

\section{Affine comparison, recognition, and open problems}

We next record the affine analogue. Let \(\mathbb A^n(\mathbb F_q)\)
denote the \(n\)-dimensional affine space over \(\mathbb F_q\). A proper
affine subspace is a translate of a proper linear subspace.

Two affine hyperplanes are called \emph{parallel} if they are translates
of the same linear hyperplane. A \emph{parallel class} is the family of
all affine hyperplanes obtained as translates of a fixed linear
hyperplane.

\begin{theorem}\label{thm:affine-max}
Let \(n\ge 1\), and let
\[
A_1,\ldots,A_m
\]
be proper affine subspaces of \(\mathbb A^n(\mathbb F_q)\), where
\(1\le m\le q\). Then
\[
|A_1\cup\cdots\cup A_m|\le m q^{n-1}.
\]
The bound is sharp. Equality holds for \(m=1\) if and only if \(A_1\) is
an affine hyperplane. For \(m\ge 2\), equality holds if and only if
\(A_1,\ldots,A_m\) are pairwise disjoint affine hyperplanes; equivalently,
they are distinct members of a parallel class of affine hyperplanes.
\end{theorem}

\begin{proof}
Every proper affine subspace of \(\mathbb A^n(\mathbb F_q)\) has at most
\(q^{n-1}\) points. Therefore
\[
|A_1\cup\cdots\cup A_m|
\le
\sum_{i=1}^m |A_i|
\le
m q^{n-1}.
\]
This proves the upper bound.

Sharpness is achieved by taking \(m\) distinct affine hyperplanes from a
single parallel class. For instance, if \(a_1,\ldots,a_m\) are distinct
elements of \(\mathbb F_q\), then
\[
x_1=a_1,\quad x_1=a_2,\quad \ldots,\quad x_1=a_m
\]
are pairwise disjoint affine hyperplanes, and their union has size
\(m q^{n-1}\).

The equality statement follows from the two inequalities above. Equality
forces each \(A_i\) to have size \(q^{n-1}\), so each \(A_i\) is an
affine hyperplane. Equality also forces
\[
|A_1\cup\cdots\cup A_m|=\sum_{i=1}^m |A_i|,
\]
so the affine hyperplanes must be pairwise disjoint. Finally, pairwise
disjoint affine hyperplanes are parallel: if two affine hyperplanes have
different directions, their defining linear equations are independent,
and the corresponding system has a solution over \(\mathbb F_q\), so the
two affine hyperplanes intersect. Therefore disjoint affine hyperplanes
must have the same direction.
\end{proof}

Taking \(m=q\) gives the affine covering number.

\begin{corollary}\label{cor:affine-cover}
The least number of proper affine subspaces needed to cover
\(\mathbb A^n(\mathbb F_q)\), where \(n\ge 1\), is \(q\). Moreover, the
covers of size \(q\) are precisely the parallel classes of affine
hyperplanes.
\end{corollary}

\begin{proof}
The \(q\) affine hyperplanes
\[
x_1=a,\qquad a\in\mathbb F_q,
\]
cover \(\mathbb A^n(\mathbb F_q)\). Conversely, by Theorem
\ref{thm:affine-max}, \(q-1\) proper affine subspaces cover at most
\[
(q-1)q^{n-1}<q^n
\]
points, so they cannot cover the affine space. If \(q\) proper affine
subspaces cover \(\mathbb A^n(\mathbb F_q)\), then equality holds in
Theorem \ref{thm:affine-max}, and hence the subspaces form a parallel
class of affine hyperplanes.
\end{proof}

\begin{remark}
The linear and affine covering numbers differ:
\[
q+1
\quad\text{for linear subspaces of } \mathbb F_q^n,
\]
but
\[
q
\quad\text{for affine subspaces of } \mathbb A^n(\mathbb F_q).
\]
The difference is geometric. Linear hyperplanes all pass through the
origin, whereas affine hyperplanes may form disjoint parallel classes.
\end{remark}

We now include a simple recognition criterion for extremal families of
hyperplanes. This is useful from a computational perspective: if the
subspaces are given by linear equations, one can check extremality by a
rank computation.

Assume \(V=\mathbb F_q^n\). A hyperplane \(H\le V\) can be written as
\[
H=\ker \ell
\]
for a nonzero linear functional \(\ell\in V^\ast\), unique up to
multiplication by a nonzero scalar.

\begin{proposition}\label{prop:recognition-max}
Let \(m\ge 2\), and let
\[
H_i=\ker \ell_i,\qquad 1\le i\le m,
\]
be distinct hyperplanes of \(V=\mathbb F_q^n\). Then
\(H_1,\ldots,H_m\) have maximum possible union size among \(m\) proper
subspaces if and only if
\[
\rank\{\ell_1,\ldots,\ell_m\}=2.
\]
Equivalently, the projective points \([\ell_1],\ldots,[\ell_m]\) lie on
one projective line in \(\mathbb P(V^\ast)\).
\end{proposition}

\begin{proof}
By Theorem \ref{thm:max-union}, an extremal family with \(m\ge 2\)
consists of distinct hyperplanes containing a common codimension-two
subspace \(U\). This is equivalent to
\[
U\subseteq \bigcap_{i=1}^m H_i
\]
with \(\codim U=2\). In dual language, the hyperplanes containing \(U\)
are precisely the kernels of nonzero linear functionals belonging to the
two-dimensional subspace
\[
U^\perp=\{\ell\in V^\ast:\ell(u)=0\text{ for all }u\in U\}.
\]
Thus the functionals \(\ell_1,\ldots,\ell_m\) span a two-dimensional
subspace of \(V^\ast\).

Conversely, if \(\rank\{\ell_1,\ldots,\ell_m\}=2\), then
\[
U=\bigcap_{i=1}^m \ker \ell_i
\]
has codimension two, and each \(H_i\) contains \(U\). Since the \(H_i\)
are distinct, Theorem \ref{thm:max-union} gives equality in the maximum
union bound.
\end{proof}

This gives the following elementary algorithm.

\begin{algorithm}[Recognition of maximum-union families]
Input: proper subspaces \(W_1,\ldots,W_m\le \mathbb F_q^n\), with
\(2\le m\le q+1\).

\begin{enumerate}
\item Check whether each \(W_i\) has dimension \(n-1\). If not, the
family is not extremal.
\item For each \(i\), write \(W_i=\ker \ell_i\) for a nonzero linear
functional \(\ell_i\).
\item Check whether the projective classes \([\ell_i]\) are distinct.
If not, the family is not extremal.
\item Compute the rank of the matrix whose rows are the coefficient
vectors of \(\ell_1,\ldots,\ell_m\).
\item The family is extremal if and only if this rank is \(2\).
\end{enumerate}
\end{algorithm}

\begin{remark}
The same rank test recognizes minimum basis-blocking families of size
\(q\). Namely, \(q\) hyperplanes form a minimum basis-blocking family if
and only if their defining linear forms are projectively distinct and
span a two-dimensional subspace of \(V^\ast\).
\end{remark}

We close with examples and open problems.

\begin{example}
Let \(V=\mathbb F_2^2\). The three one-dimensional subspaces are
\[
\mathbb F_2(1,0),\qquad
\mathbb F_2(0,1),\qquad
\mathbb F_2(1,1).
\]
For \(m=1,2,3\), Theorem \ref{thm:max-union} gives
\[
2^{0}\bigl(1+m(2-1)\bigr)=1+m.
\]
Thus one proper subspace covers at most \(2\) vectors, two proper
subspaces cover at most \(3\) vectors, and three proper subspaces cover
all \(4\) vectors.
\end{example}

\begin{example}
Let \(V=\mathbb F_3^2\). The four one-dimensional subspaces form the
pencil of lines through the origin. For \(m=1,2,3,4\), Theorem
\ref{thm:max-union} gives
\[
3^{0}\bigl(1+m(3-1)\bigr)=1+2m.
\]
Thus \(m\) lines through the origin cover at most
\[
3,\ 5,\ 7,\ 9
\]
points, respectively. The maximum is achieved by any \(m\) distinct lines
through the origin.
\end{example}

\begin{example}
Let \(V=\mathbb F_2^3\). Fix a one-dimensional subspace \(U\le V\). The
three planes containing \(U\) form a pencil. For \(m=1,2,3\), Theorem
\ref{thm:max-union} gives
\[
2^{1}\bigl(1+m(2-1)\bigr)=2+2m.
\]
Thus one plane covers \(4\) vectors, two planes from the pencil cover
\(6\) vectors, and all three planes cover \(8\) vectors.
\end{example}

\begin{problem}
Fix \(1\le r<n\). Determine the maximum size of a union of \(m\)
\(r\)-dimensional linear subspaces of \(\mathbb F_q^n\), and classify the
equality cases.
\end{problem}

\begin{problem}
Give stability versions of Theorem \ref{thm:max-union}. For example, if
\(m\le q+1\) proper subspaces have union size close to
\[
q^{n-2}\bigl(1+m(q-1)\bigr),
\]
must most of them be close to a common pencil of hyperplanes?
\end{problem}

\begin{problem}
Develop a matroidal analogue of basis-blocking families. In a vector
space, bases are the bases of the associated linear matroid, and proper
subspaces correspond to flats. For which matroids can one determine the
least number of proper flats whose union meets every basis?
\end{problem}

\begin{problem}
Study algorithmic recognition for near-extremal families. Given
subspaces \(W_1,\ldots,W_m\le \mathbb F_q^n\), how efficiently can one
decide whether their union is within a prescribed error of the maximum
possible size?
\end{problem}

\begin{remark}
The results in this paper show that several extremal covering questions
for finite vector spaces are controlled by the same geometric object:
pencils of hyperplanes through codimension-two subspaces. The maximum
partial covers, minimum full covers, sharp one-step defects, and minimum
basis-blocking families all arise from this single configuration.
\end{remark}
\bigskip

\noindent \textbf{Declarations}

\medskip
\noindent\textbf{\small Funding:} {\small The research of Mohsen Aliabadi is partially supported by an AMS--Simons Research Enhancement Grant (award no.~SFI-MPS-T-Institutes-00021934).} 

\medskip
\noindent \textbf{\small Conflict of interest:} {\small The authors declare that they have no conflict of interest.}

\medskip
\noindent \textbf{\small Data availability statement:} {\small Data sharing is not applicable to this article as no datasets were generated or analyzed during the current study.}


\end{document}